\documentclass[12pt]{amsart}

\usepackage{amsthm,amssymb,fullpage}
\usepackage[usenames]{color}

\begin{document}

\numberwithin{equation}{section}

\def\comment#1#2{\textcolor{blue}{(#1: #2)}}
\def\red#1{\textcolor{red}{#1}}

\theoremstyle{plain}
\newtheorem{theorem}{Theorem}[section]
\newtheorem{conjecture}[theorem]{Conjecture}
\newtheorem{corollary}[theorem]{Corollary}
\newtheorem{definition}[theorem]{Definition}
\newtheorem{lemma}[theorem]{Lemma}
\newtheorem{proposition}[theorem]{Proposition}
\newtheorem{problem}[theorem]{Problem}

\theoremstyle{definition}
\newtheorem{example}[theorem]{Example}
\newtheorem{remark}[theorem]{Remark}
\newtheorem{algorithm}[theorem]{Algorithm}
\newtheorem{construction}[theorem]{Construction}

\def\ld{\backslash}
\def\ldd{\backslash^\cdot}
\def\rdd{/^\cdot}
\def\im#1{\mathrm{Im}#1}
\def\ker#1{\mathrm{Ker}#1}
\def\aut#1{\mathrm{Aut}#1}
\def\aff#1{\mathrm{Aff}#1}
\def\img#1{\mathrm{Im}#1}
\def\lmlt#1{\mathrm{LMlt}#1}
\def\rmlt#1{\mathrm{RMlt}#1}
\def\mlt#1{\mathrm{Mlt}#1}
\def\inn#1{\mathrm{Inn}#1}
\def\dis#1{\mathrm{Dis}#1}
\def\Z{\mathbb Z}
\def\F{\mathbb F}
\def\Q{\mathcal Q}
\def\CA{\mathrm{Aut}_C}
\def\ort#1{\mathrm{Ort}(#1)}
\def\CO#1{J\mathrm{Aut}_C(#1)\cap\ort{#1}}

\title{Distributive and trimedial quasigroups of order 243}

\author{P\v remysl Jedli\v cka}
\author{David Stanovsk\'y}
\author{Petr Vojt\v echovsk\'y}

\address[Jedli\v cka]{Department of Mathematics, Faculty of Technology, Czech University of Life Sciences, Prague, Czech Republic}
\address[Stanovsk\'y]{Department of Algebra, Faculty of Mathematics and Physics, Charles University, Prague, Czech Republic}
\address[Stanovsk\'y]{Department of Electrotechnics and Computer Science, Kazakh-British Technical University, Almaty, Kazakhstan}
\address[Vojt\v{e}chovsk\'y]{Department of Mathematics, University of Denver, 2280 S Vine St, Denver, Colorado 80208, U.S.A.}
\email[Jedli\v cka]{jedlickap@tf.czu.cz}
\email[Stanovsk\'y]{stanovsk@karlin.mff.cuni.cz}
\email[Vojt\v{e}chovsk\'y]{petr@math.du.edu}

\thanks{Research partially supported by the GA\v CR grant 13-01832S (Stanovsk\'y), the Simons Foundation Collaboration Grant 210176 (Vojt\v{e}chovsk\'y), and the University of Denver PROF grant (Vojt\v{e}chovsk\'y).}



\begin{abstract}
We enumerate three classes of non-medial quasigroups of order $243=3^5$ up to isomorphism. There are $17004$ non-medial trimedial quasigroups of order $243$ (extending the work of Kepka, B\'en\'eteau and Lacaze), $92$ non-medial distributive quasigroups of order $243$ (extending the work of Kepka and N\v{e}mec), and $6$ non-medial distributive Mendelsohn quasigroups of order $243$ (extending the work of Donovan, Griggs, McCourt, Opr\v sal and Stanovsk\'y).

The enumeration technique is based on affine representations over commutative Moufang loops, on properties of automorphism groups of commutative Moufang loops, and on computer calculations with the \texttt{LOOPS} package in \texttt{GAP}.
\end{abstract}

\maketitle

\section{Introduction}\label{sec:intro}

Enumeration of quasigroups (equivalently, latin squares) is one of the classical topics of combinatorics. Enumerating all quasigroups of a given order $n$ is a difficult problem already for small values of $n$. Indeed, the number of latin squares is known only up to $n=11$ \cite{McKWan}, and the number of quasigroups up to isomorphism is known only up to $n=10$ \cite{McKMeyMyr}. Consequently, many quasigroup enumeration projects deal with particular well-studied classes or varieties.

In this paper we focus on quasigroups that admit an affine representation over nonassociative commutative Moufang loops. We enumerate non-medial trimedial quasigroups of order $243=3^5$ up to isomorphism. In particular, we enumerate non-medial distributive quasigroups and non-medial distributive Mendelsohn quasigroups of order $243$, the latter algebraic structures being in one-to-one correspondence with non-affine distributive Mendelsohn triple systems of order $243$.

The enumeration of quasigroups affine over nonassociative commutative Moufang loops is interesting only for orders that are powers of $3$ (see below). The previous step, $n=81=3^4$, has been completed in $1981$ by Kepka and N\v{e}mec for distributive quasigroups \cite{KN}, and in $1987$ by Kepka, B\'en\'eteau and Lacaze for trimedial quasigroups \cite{KBL}.
Our calculations independently verify their enumeration results.

\medskip

A \emph{quasigroup} is a set $Q$ with a binary operation $+$ such that all left translations $L_x:Q\to Q$, $y\mapsto x+y$ and all right translations $R_x:Q\to Q$, $y\mapsto y+x$ are bijections of $Q$. A quasigroup $(Q,+)$ is a \emph{loop} if it possesses a neutral element, that is, an element $0$ satisfying $0+x=x+0=x$ for all $x\in Q$.

A quasigroup $(Q,+)$ is called \emph{idempotent} if it satisfies the identity \begin{displaymath}
    x+x=x,
\end{displaymath}
\emph{medial} (also \emph{entropic} or \emph{abelian}) if it satisfies the identity
\begin{displaymath}
    (x+y)+(u+v)=(x+u)+(y+v),
\end{displaymath}
and \emph{distributive} if it satisfies the two identities
\begin{align*}
    x+(y+z) &=(x+y)+(x+z),\\
    (x+y)+z &=(x+z)+(y+z).
\end{align*}
A quasigroup $(Q,+)$ is \emph{trimedial} (also \emph{terentropic} or \emph{triabelian}) if every three elements of $Q$ generate a medial subquasigroup.
Belousov established the following connection between these types of quasigroups:

\begin{theorem}[\cite{Bel-dq}]\label{Th:Belousov}
A quasigroup is distributive if and only if it is trimedial and idempotent.
\end{theorem}

Historically, distributive and medial quasigroups were one of the first nonassociative algebraic structures studied \cite{BM}. Their structure theory has been developed mostly in the 1960s and 1970s; see \cite{Ben-dq} or \cite[Section 3]{Sta-latin} for an overview. Quasigroups satisfying various forms of self-distributivity were one of the favorite topics of Belousov's school of quasigroup theory \cite{Bel}, and they have connections to other branches of mathematics as well \cite[Section 1]{Sta-latin}.


The classification of medial quasigroups is to a large extent a matter of understanding conjugation in the automorphism groups of abelian groups. This is explained in detail in \cite{SV}, for instance, where one can also find the complete classification of medial quasigroups up to order 63 (up to order 127 with a few gaps). Hou \cite{Hou} has stronger results on the enumeration of idempotent medial quasigroups.

In the present paper, we will focus on \emph{non-medial} trimedial quasigroups, which will require computational tools that are quite different from those of the medial case.

\medskip

One of the fundamental tools in quasigroup theory is loop isotopy. In particular, affine representations of quasigroups over various classes of loops are  tremendously useful in the study of quasigroups.
The Kepka theorem \cite{Kep} (see Theorem \ref{Th:Kepka}) represents trimedial quasigroups over commutative Moufang loops. It is a generalization of both the Toyoda-Murdoch-Bruck theorem \cite{BruckQ,Murdoch,Toyoda} (see Theorem \ref{Th:ToyodaBruck}) that represents medial quasigroups over abelian groups, and the Belousov-Soublin theorem \cite{Bel-dq,Sou} (see Theorem \ref{Th:BelousovSoublin}) that represents distributive quasigroups over commutative Moufang loops. Theorem \ref{Th:Kepka-iso}, proved in \cite{KBL}, solves the isomorphism problem for representations of trimedial quasigroups and forms the basis for our enumeration algorithm.
A detailed account on these representation theorems can be found in \cite{Sta-latin}.

\medskip

The class of commutative Moufang loops has attracted attention from the very onset of abstract loop theory.
A significant part of the fundamental text of loop theory, Bruck's \emph{``A survey of binary systems''} \cite{Bru}, has been written to develop tools for dealing with commutative Moufang loops.

Every finite commutative Moufang loop decomposes as a direct product of an abelian group of order coprime to 3 and of a commutative Moufang loop of order a power of $3$ \cite[Theorem 7C]{BruckTAMS}.
It was known to Bruck that there are no nonassociative commutative Moufang loops of order less than $3^4$. Kepka and N\v{e}mec \cite{KN} classified nonassociative commutative Moufang loops of orders $3^4$ and $3^5$ up to isomorphism: there are two of order $3^4$ and six of order $3^5$. See \cite{KN} for explicit constructions of these commutative Moufang loops, and \cite[Theorem IV.3.44]{Ben} for more results on commutative Moufang loops with a prescribed nilpotence class.

Every automorphism of a commutative Moufang loop decomposes as a direct product of automorphisms of the two coprime components. Therefore, thanks to Kepka's theorem, every finite non-medial trimedial quasigroup is a direct product of a medial quasigroup of order coprime to 3 and of a non-medial trimedial quasigroup of order a power of $3$, and there are no non-medial trimedial quasigroups of order less than $3^4$.

The classification of non-medial distributive quasigroups of order $3^4$ was also carried out in \cite{KN}: there are 6 such quasigroups up to isomorphism. Non-medial trimedial quasigroups of order $3^4$ were enumerated by Kepka, B\'en\'etau and Lacaze in \cite{KBL}: there are 35 of them up to isomorphism. Both \cite{KN} and \cite{KBL} use affine representations and a careful analysis of the automorphism groups of the two nonassociative commutative Moufang loops of order $3^4$, without using computers.

\medskip

The main result of this paper is a computer enumeration of non-medial distributive quasigroups and non-medial trimedial quasigroups of order $3^5$ up to isomorphism; see Table \ref{t:main2}. The paper is organized as follows.

In Section \ref{Sc:Representation} we summarize theoretical results that we use in the enumeration. We state the representation theorems and the isomorphism theorem, introduce the notions of $J$-central automorphisms and orthomorphisms, and finish the section with notes on representations of distributive Steiner and Mendelsohn quasigroups. Most of the contents of Section \ref{Sc:Representation} are present, implicitly or explicitly, in \cite{DGMOS,KBL,KN}.

In Section \ref{Sc:Classification} we describe in detail our main contribution, the classification algorithm (Theorem \ref{Th:Alg}).

In Section \ref{Sc:Results} we present the results of our calculations; see Tables \ref{t:main1} and \ref{t:main2}. We also give a sample of explicit constructions of non-medial distributive quasigroups of order $3^5$, including all those from which one can recover the non-affine distributive Mendelsohn triple systems of order~$3^5$. At the end, we discuss the phenomenon that for many small commutative Moufang loops all central automorphisms commute.

\subsection*{Basic definitions and results}

Loops will be denoted additively, assuming implicitly $Q=(Q,+,0)$. The \emph{center} $Z(Q)$ of a loop $Q$ is the set of all elements of $Q$ that commute and associate with all elements of~$Q$. The \emph{associator subloop} $A(Q)$ of a loop $Q$ is the smallest normal subloop of $Q$ generated by all associators $L^{-1}_{x+(y+z)}((x+y)+z)$. The \emph{automorphism group} of $Q$ will be denoted by $\aut(Q)$.

A loop $Q$ is said to have \emph{two-sided inverses} if for every $x\in Q$ there is $-x\in Q$ such that $x+(-x)=0=(-x)+x$. We then write $x-y$ as a shorthand for $x+(-y)$, and we define $J$ to be the inversion mapping
\begin{displaymath}
    J:Q\to Q,\quad x\mapsto -x.
\end{displaymath}
Clearly, $J$ is a permutation of $Q$ that commutes with all automorphisms of $Q$.

A loop $Q$ is \emph{power associative} if any element of $Q$ generates an associative subloop. A loop $Q$ is \emph{diassociative} if any two elements of $Q$ generate an associative subloop.

A loop $Q$ with two-sided inverses has the \emph{automorphic inverse property} if the inversion mapping $J$ is an automorphism, that is, if $-(x+y) = -x-y$ holds for every $x$, $y\in Q$. Note that if $Q$ has the automorphic inverse property then $J\in Z(\aut(Q))$. Commutative diassociative loops obviously satisfy the automorphic inverse property.

A loop $Q$ is \emph{Moufang} it it satisfies the identity $x+(y+(x+z)) = ((x+y)+x)+z$. By Moufang's theorem \cite{Moufang}, Moufang loops are diassociative. In a commutative Moufang loop $Q$, we have $x+x+x=3x\in Z(Q)$ for every $x\in Q$ \cite{BruckTAMS}. See \cite{Bel,Ben,BruckTAMS,Bru} for more results on commutative Moufang loops.

\section{Affine representation of trimedial quasigroups} \label{Sc:Representation}

\subsection{Affine representation and isomorphism theorem}\label{Ssec:representation}

In group theory, an automorphism $\alpha$ of a group $G=(G,+)$ is said to be \emph{central} if it commutes with all inner automorphisms of $G$. Equivalently, $\alpha\in\aut(G)$ is central if $Z(G)+\alpha(x)=Z(G)+x$ for every $x\in G$. It is well known that the set of all central automorphisms of $G$ forms a normal subgroup of $\aut(G)$. We generalize these concepts and results to loops as follows:

\begin{definition}
Let $Q$ be a loop and $\alpha:Q\to Q$ a mapping. Then $\alpha$ is said to be \emph{central} if $Z(Q)+\alpha(x)=Z(Q)+x$ for every $x\in Q$.
The set of all central automorphisms of $Q$ will be denoted by $\CA(Q)$.
\end{definition}

Note that if $Q$ is an abelian group then all mappings $\alpha:Q\to Q$ are central.

\begin{lemma}\label{Lm:CAut}
Let $Q$ be a loop. Then $\CA(Q)$ is a normal subgroup of $\mathrm{Aut}(Q)$.
\end{lemma}
\begin{proof}
If $\alpha$, $\beta\in\CA(Q)$ then $Z(Q)+\alpha\beta(x) = Z(Q)+\beta(x)=Z(Q)+x$ and $Z(Q)+x = Z(Q)+\alpha\alpha^{-1}(x) = Z(Q)+\alpha^{-1}(x)$, so $\alpha\beta\in\CA(Q)$ and $\alpha^{-1}\in\CA(Q)$. If further $\gamma\in\mathrm{Aut}(Q)$, then $Z(Q)+\gamma^{-1}\alpha\gamma(x) = \gamma^{-1}(Z(Q)+\alpha\gamma(x)) = \gamma^{-1}(Z(Q) + \gamma(x)) = Z(Q)+x$, so $\gamma^{-1}\alpha\gamma\in\CA(Q)$.
\end{proof}

\begin{lemma}\label{Lm:CentralTSI}
Let $Q$ be a loop with two-sided inverses and $\alpha:Q\to Q$ a mapping. Then $\alpha$ is central if and only if $x-\alpha(x)\in Z(Q)$ for every $x\in Q$.
\end{lemma}
\begin{proof}
Since the elements of $Z(Q)$ associate with all elements of $Q$, the following conditions are equivalent: $Z(Q)+x=Z(Q)+\alpha(x)$, $Z(Q)+x-\alpha(x) = Z(Q)$, $x-\alpha(x)\in Z(Q)$.
\end{proof}

We will now show how to represent trimedial quasigroups over commutative Moufang loops. We start with a general definition.

\begin{definition}\label{Df:Affine}
Let $(Q,+)$ be a loop, let $\varphi$, $\psi$ be automorphisms of $(Q,+)$, and let $c\in Z(Q,+)$. Define a binary operation $*$ on $Q$ by
\begin{equation}\label{Eq:Affine}
    x*y=\varphi(x)+\psi(y)+c.
\end{equation}
The resulting quasigroup $(Q,*)$ is said to be \emph{affine over the loop $(Q,+)$}, it will be denoted by $\Q(Q,+,\varphi,\psi,c)$, and the quintuple $(Q,+,\varphi,\psi,c)$ will be called an \emph{arithmetic form} of $(Q,*)$.
\end{definition}

\begin{remark}
Definition \ref{Df:Affine} can be generalized in various ways, for instance by setting $x*y = (\varphi(x)+c)+(\psi(y)+d)$ for automorphisms $\varphi$, $\psi$ and arbitrary elements $c$, $d$. On the other hand, it can be specialized by assuming that $c=0$, that $\varphi\psi=\psi\varphi$, that the automorphisms $\varphi$, $\psi$ are central, etc. See \cite[Section 2.3]{Sta-latin} for a detailed discussion.
\end{remark}

\begin{lemma}\label{Lm:AffineIdempotent}
An affine quasigroup $(Q,*)=\mathcal Q(Q,+,\varphi,\psi,c)$ is idempotent if and only if $c=0$ and $\varphi+\psi=id$.
\end{lemma}
\begin{proof}
If $(Q,*)$ is idempotent then $x=x*x=\varphi(x)+\psi(x)+c$ for every $x\in Q$. With $x=0$ we deduce $c=0$. Then $\varphi(x)+\psi(x)=x$ for every $x\in Q$, so $\varphi+\psi=id$.

Conversely, if $\varphi+\psi=id$ and $c=0$ then $x*x=\varphi(x)+\psi(x)+c = x$.
\end{proof}

\begin{lemma}\label{Lm:AffineMedial}
An affine quasigroup $(Q,*)=\mathcal Q(Q,+,\varphi,\psi,c)$ is medial if and only if $(Q,+)$ is an abelian group and $\varphi\psi=\psi\varphi$.
\end{lemma}
\begin{proof}
Note that
\begin{displaymath}
    (x*u)*(v*y) = (\varphi\varphi(x)+\varphi\psi(u)+\varphi(c))
        + (\psi\varphi(v)+\psi\psi(y)+\psi(c)) + c.
\end{displaymath}
Since $\varphi(c)$, $\psi(c)$, $c$ are central, we see that $(Q,*)$ is medial if and only if
\begin{equation}\label{Eq:Medial}
    (\varphi\varphi(x)+\varphi\psi(u))+(\psi\varphi(v)+\psi\psi(y)) =
    (\varphi\varphi(x)+\varphi\psi(v))+(\psi\varphi(u)+\psi\psi(y)).
\end{equation}
If $(Q,+)$ is an abelian group and $\varphi\psi=\psi\varphi$ then \eqref{Eq:Medial} holds.

Conversely, suppose that \eqref{Eq:Medial} holds. With $x=y=v=0$ we deduce $\varphi\psi=\psi\varphi$ from \eqref{Eq:Medial}. Then with $x=y=0$ we deduce $\varphi\psi(u)+\varphi\psi(v) = \varphi\psi(v)+\varphi\psi(u)$, so $(Q,+)$ is commutative. Finally, with $x=0$ we deduce the identity $r+(s+t)=s+(r+t)$, which, combined with commutativity, yields associativity of $(Q,+)$.
\end{proof}

Let us now state the representation theorem that forms a basis for our enumeration algorithm.

\begin{definition}
Let $(Q,+)$ be a loop with two-sided inverses. We say that a quasigroup $(Q,*)$ is \emph{centrally affine over $(Q,+)$} if it admits an arithmetic form $(Q,+,\varphi,\psi,c)$ as in Definition $\ref{Df:Affine}$ such that $-\varphi$, $-\psi$ are central mappings of $(Q,+)$. We then call $(Q,+,\varphi,\psi,c)$ a \emph{central arithmetic form}.
\end{definition}

When $(Q,+)$ is an abelian group then there is no distinction between arithmetic forms and central arithmetic forms, and we will not use the adjective ``central''.

\begin{theorem}[Kepka \cite{Kep}]\label{Th:Kepka}
A quasigroup is trimedial if and only if it admits a central arithmetic form $(Q,+,\varphi,\psi,c)$, where $(Q,+)$ is a commutative Moufang loop and $\varphi\psi=\psi\varphi$.
\end{theorem}

Lemmas \ref{Lm:AffineIdempotent}, \ref{Lm:AffineMedial} and Theorem \ref{Th:Belousov} show that the Kepka theorem generalizes both the Toyoda-Murdoch-Bruck theorem (set $(Q,+)$ to be an abelian group) and the Belousov-Soublin theorem (set $c=0$ and $\varphi+\psi=id$):

\begin{theorem}[Toyoda-Murdoch-Bruck \cite{BruckQ, Murdoch, Toyoda}]\label{Th:ToyodaBruck}
A quasigroup is medial if and only if it admits an arithmetic form $(Q,+,\varphi,\psi,c)$, where $(Q,+)$ is an abelian group and $\varphi\psi=\psi\varphi$.
\end{theorem}

\begin{theorem}[Belousov-Soublin \cite{Bel-dq, Sou}]\label{Th:BelousovSoublin}
A quasigroup is distributive if and only if it admits a central arithmetic form $(Q,+,\varphi,\psi,0)$, where $(Q,+)$ is a commutative Moufang loop and $\varphi = id-\psi$.
\end{theorem}

Note that in the Belousov-Soublin theorem, we have $\varphi\psi = (id-\psi)\psi =\psi-\psi^2 = \psi(id-\psi)=\psi\varphi$ for free.

We now present a solution to the isomorphism problem for centrally affine quasigroups that covers the representations in Theorems \ref{Th:Kepka}, \ref{Th:ToyodaBruck}, \ref{Th:BelousovSoublin}.

\begin{theorem}[\cite{KBL}]\label{Th:Kepka-iso}
Let $(Q_1,+_1)$, $(Q_2,+_2)$ be commutative Moufang loops. Two centrally affine quasigroups $\mathcal Q(Q_1,+_1,\varphi_1,\psi_1,c_1)$, $\mathcal Q(Q_2,+_2,\varphi_2,\psi_2,c_2)$ are isomorphic if and only if there is a loop isomorphism $f:(Q_1,+_1)\to(Q_2,+_2)$ and $u\in\mathrm{Im}(id-_1(\varphi_1+_1\psi_1))$ such that
\[\varphi_2 = f \varphi_1 f^{-1},\ \psi_2 = f \psi_1 f^{-1}\text{ and }c_2=f(c_1+_1u).\]
\end{theorem}

\begin{remark}
The isomorphism test condition of Theorem \ref{Th:Kepka-iso} is stated differently in \cite{KBL}, namely as: \emph{There is a loop isomorphism $f:(Q_1,+_1)\to(Q_2,+_2)$ and $w\in Q_2$ such that}
\begin{displaymath}
    \varphi_2f = f \varphi_1,\quad \psi_2f = f \psi_1,\quad f(c_1)-_2c_2 = w-_2(\varphi_2(w)+_2\psi_2(w)).
\end{displaymath}
We claim that this condition is equivalent to the condition of Theorem \ref{Th:Kepka-iso}. First, because ``to be isomorphic'' is a symmetric relation, we can replace the above condition with: \emph{There is a loop isomorphism $f:(Q_2,+_2)\to (Q_1,+_1)$ and $w\in Q_1$ such that}
\begin{displaymath}
    \varphi_1f=f\varphi_2,\quad \psi_1f = f\psi_2,\quad f(c_2)-_1c_1 = w-_1(\varphi_1(w)+_1\psi_1(w)).
\end{displaymath}
Upon considering $f^{-1}$, we can further replace it with the statement: \emph{There is a loop isomorphism $f:(Q_1,+_1)\to (Q_2,+_2)$ and $w\in Q_1$ such that}
\begin{displaymath}
    \varphi_1 f^{-1}=f^{-1}\varphi_2,\quad \psi_1 f^{-1} = f^{-1}\psi_2,\quad f^{-1}(c_2)-_1c_1 = w-_1(\varphi_1(w)+_1\psi_1(w)).
\end{displaymath}
The condition on $c_2$ is then equivalent to $c_2 = f(c_1+_1 w -_1 (\varphi_1(w)+\psi_1(w)))$, which says that $c_2 = f(c_1+_1 u)$ for some $u\in \mathrm{Im}(id-_1(\varphi_1+_1\psi_1))$.
\end{remark}

Note that in the distributive case ($c=0$ and $\varphi+\psi=id$), the isomorphism test of Theorem \ref{Th:Kepka-iso} reduces to: \emph{There is a loop isomorphism $f:(Q_1,+_1)\to (Q_2,+_2)$ such that $\psi_2 = f\psi_1 f^{-1}$.}


\subsection{$J$-central mappings}\label{Ssec:J-centrality}

\begin{definition}
Let $Q$ be a loop, $\xi$ a permutation of $Q$ and $\alpha:Q\to Q$ a mapping. We say that $\alpha$ is \emph{$\xi$-central} if $\xi^{-1}\alpha$ is central.
\end{definition}

Observe the following:

\begin{lemma}
Let $Q$ be a loop and $\alpha$, $\xi$ automorphisms of $Q$. Then $\alpha$ is $\xi$-central if and only if $\alpha$ belongs to the coset $\xi\CA(Q)$.
\end{lemma}

\begin{corollary}\label{Cr:JCentral}
Let $Q$ be a loop with the automorphic inverse property and $J$ the inversion mapping. Then the coset $J\CA(Q)$ is the set of all $J$-central mappings of $Q$.
\end{corollary}

\begin{remark}
$J$-central mappings were called \emph{$1$-central} in earlier papers  \cite{Kep,KN,Sta-latin}.
\end{remark}

We now give another useful characterization of $J$-central mappings.

For a loop $Q$ and a mapping $\alpha:Q\to Q$, let $\hat\alpha$ denote the mapping $id+\alpha$, that is,
\begin{displaymath}
    \hat\alpha:Q\to Q,\quad x\mapsto x+\alpha(x).
\end{displaymath}
If $Q$ has two-sided inverses, we have $\alpha(x)=-x+\hat\alpha(x)$.

\begin{lemma}\label{Lm:JCentral}
Let $Q$ be a loop with two-sided inverses and $\alpha:Q\to Q$ a mapping. Then $\alpha$ is $J$-central if and only if $\hat\alpha(x)\in Z(Q)$ for every $x\in Q$.
\end{lemma}
\begin{proof}
Note that $J^{-1}=J$. The following statements are equivalent: $\alpha$ is $J$-central, $J\alpha$ is central, $x-J\alpha(x)\in Z(Q)$ for every $x\in Q$ (by Lemma \ref{Lm:CentralTSI}), $\hat\alpha(x)=x+\alpha(x)\in Z(Q)$ for every $x\in Q$.
\end{proof}

A stronger equivalence holds for endomorphisms:

\begin{lemma}\label{Lm:alphahat}
Let $Q$ be a loop with the automorphic inverse property and let $\alpha:Q\to Q$ be a mapping. Then $\alpha$ is a $J$-central endomorphism if and only if $\hat\alpha$ is an endomorphism into $Z(Q)$. Moreover,
\begin{displaymath}
    \ker(\alpha) = \{x\in Q:\alpha(x)=0\} = \{x\in Q:\hat\alpha(x)=x\} = \mathrm{Fix}(\hat\alpha).
\end{displaymath}
\end{lemma}

\begin{proof}
Throughout the proof, we will use Lemma \ref{Lm:JCentral} without reference.
Suppose that $\alpha$ is a $J$-central endomorphism. Then $\hat\alpha(x+y) = (x+y)+\alpha(x+y) = (x+y)+(\alpha(x)+\alpha(y)) = (x+y)+((-x+\hat\alpha(x))+(-y+\hat\alpha(y))) = (x+y)+(-x-y)+\hat\alpha(x)+\hat\alpha(y) =
(x+y)-(x+y)+\hat\alpha(x)+\hat\alpha(y) =
\hat\alpha(x)+\hat\alpha(y)$, where we have used the automorphic inverse property.

Conversely, suppose that $\hat\alpha$ is an endomorphism into $Z(Q)$. Then $\alpha(x+y) = -(x+y)+\hat\alpha(x+y) = (-x-y) + \hat\alpha(x)+\hat\alpha(y) = (-x+\hat\alpha(x))+(-y+\hat\alpha(y)) = \alpha(x)+\alpha(y)$, where we have again used the automorphic inverse property.

To finish the proof, note that $\alpha(x)=0$ if and only if $\hat\alpha(x)=x$.
\end{proof}

In particular, if $Q$ is a finite loop with the automorphic inverse property and $\alpha:Q\to Q$ is a mapping, then $\alpha$ is a $J$-central automorphism if and only if $\hat\alpha$ is an endomorphism into $Z(Q)$ with a unique fixed point.

\subsection{Orthomorphisms and orthoautomorphisms}

We say that a permutation $\alpha$ of a loop $Q$ with two-sided inverses is a (\emph{left}) \emph{orthomorphism} if the mapping $id-\alpha$ is also a permutation of $Q$. The set of all orthomorphisms of $Q$ will be denoted $\ort{Q}$.

\begin{remark}
Orthomorphisms were originally defined in \cite{JDM} for finite groups. Researchers now routinely work with orthomorphisms in arbitrary groups, but usually use the dual notion of a right orthomorphism ($-id+\alpha$ is a permutation). In loops with the automorphic inverse property, $id-\alpha$ is a permutation if and only if $-id+\alpha$ is a permutation, so there is no distinction between left and right orthomorphisms.
\end{remark}

An orthomorphism need not be an automorphism. For brevity, we call orthomorphisms that are also automorphisms \emph{orthoautomorphisms}. Thus $\CO{Q}$ is the set of all $J$-central orthoautomorphisms in any loop with the automorphic inverse property, cf. Corollary \ref{Cr:JCentral}.

\begin{lemma}\label{Lm:ortho}
Let $Q$ be a commutative Moufang loop and let $\alpha:Q\to Q$ be a mapping. Then $\alpha\in\CO{Q}$ if and only if $id-\alpha\in\CO{Q}$.
\end{lemma}
\begin{proof}
Let $D=\CO{Q}$. In any diassociative loop we have $id-(id-\alpha)=\alpha$ because $x-(x-\alpha(x)) = \alpha(x)$. It therefore suffices to show that if $\alpha\in D$ then $\beta=id-\alpha\in D$. Suppose that $\alpha\in D$. By Lemma \ref{Lm:alphahat}, $\hat\alpha$ is an endomorphism into $Z(Q)$.

For every $x\in Q$ we have $\beta(x) = x-\alpha(x) = x-(-x+\hat\alpha(x)) = 2x-\hat\alpha(x)$. Hence $\beta(x)+\beta(y) = (2x-\hat\alpha(x))+(2y-\hat\alpha(y)) = (2x+2y)-(\hat\alpha(x)+\hat\alpha(y)) = 2(x+y)-\hat\alpha(x+y) = \beta(x+y)$, proving that $\beta\in\aut(Q)$. We also have $\hat\beta(x)=x+\beta(x)=3x-\hat\alpha(x)\in Z(Q)$ because $3x\in Z(Q)$, so $\beta$ is $J$-central by Lemma \ref{Lm:JCentral}. Finally, $id-\beta = id-(id-\alpha)=\alpha$ shows that $\beta$ is an orthomorphism.
\end{proof}

\begin{lemma}\label{Lm:conj}
Let $Q$ be a loop with two-sided inverses. Then the subsets $J\CA(Q)$ and $\CO{Q}$ of $\aut(Q)$ are closed under conjugation by elements of $\aut(Q)$.
\end{lemma}
\begin{proof}
The first claim follows from the fact that $\CA(Q)$ is a normal subgroup of $\aut(Q)$ (see Lemma \ref{Lm:CAut}) and that $J$ commutes with all automorphisms of $Q$.

If $\alpha$ is an orthomorphism then $id-\alpha$ is a permutation of $Q$, hence $id-\alpha^\xi=(id-\alpha)^\xi$ is a permutation of $Q$ for any $\xi\in\aut(Q)$, and $\alpha^\xi$ is an orthomorphism.
\end{proof}

Here is a useful variation of the Belousov-Soublin theorem which will be used in Section~\ref{Sc:Results}:

\begin{proposition}\label{Cr:BelousovSoublin}
A quasigroup $(Q,*)$ is distributive if and only if there is a commutative Moufang loop $(Q,+)$ and a $J$-central orthoautomorphism $\psi$ of $(Q,+)$ such that
\[x*y = (2x-y)+\hat\psi(y-x).\]
\end{proposition}
\begin{proof}
By Lemma \ref{Lm:ortho}, if $(Q,+)$ is a commutative Moufang loop and $\psi\in\CO{Q,+}$ then $id-\psi\in\CO{Q,+}\subseteq J\CA(Q,+)$.
In view of Theorem \ref{Th:BelousovSoublin}, it remains to show that $(id-\psi)(x)+\psi(y) = (2x-y)+\hat\psi(y-x)$. By Lemma \ref{Lm:alphahat}, $\hat\psi$ is an endomorphism into $Z(Q,+)$. Therefore, $(x-\psi(x))+\psi(y) = (2x-\hat\psi(x)) + (-y+\hat\psi(y)) = (2x-y) + \hat\psi(y)-\hat\psi(x) = (2x-y) + \hat\psi(y-x)$.
\end{proof}

\subsection{Quasigroups corresponding to triple systems}\label{Ssc:triplesystems}

Certain distributive quasigroups correspond to interesting combinatorial designs.

A \emph{Steiner triple system} is a pair $(V,B)$, where $V$ is a set and $B$ is a collection of $3$-element subsets of $V$ such that for every distinct $x$, $y\in V$ there is a unique $z\in V$ such that $\{x,y,z\}\in B$ \cite{Kirkman}.

A \emph{Hall triple system} is a Steiner triple system $(V,B)$ such that for every $x\in V$ there exists an involutory automorphism of $(V,B)$ whose only fixed point is $x$ \cite{Hall}.

A \emph{Mendelsohn triple system} is a pair $(V,B)$, where $V$ is a set and $B$ is a collection of cyclically ordered triples $\langle x,y,z\rangle = \langle y,z,x\rangle = \langle z,x,y\rangle$ of distinct elements of $V$ such that for any ordered tuple $(x,y)$ of distinct elements of $V$ there is a unique $z\in V$ such that $\langle x,y,z\rangle\in B$ \cite{Mendelsohn}.

Given a Steiner or Mendelsohn triple system $(V,B)$, respectively, we can define a quasigroup operation on $V$ as follows: if $x=y$, let $x*y=x$, otherwise let $x*y=z$, where $z$ is the unique element of $V$ such that $\{x,y,z\}\in B$, respectively $\langle x,y,z\rangle\in B$. There is a one-to-one correspondence between Hall triple systems and distributive Steiner quasigroups, and between distributive Mendelsohn triple systems and distributive Mendelsohn quasigroups; see \cite{DGMOS} for details.
The following simple criterion identifies the relevant quasigroups in our classification results.

\begin{proposition}[{\cite[Proposition 2.1]{DGMOS}}]
Let $Q=(Q,+)$ be a commutative Moufang loop and let $\psi\in\CO{Q}$. The corresponding distributive quasigroup $\mathcal Q(Q,+,id-\psi,\psi,0)$ is:
\begin{itemize}
\item[(i)] Steiner if and only if $Q$ has exponent $3$ and $\psi(x)=-x$ for every $x\in Q$;
\item[(ii)] Mendelsohn if and only if $\psi^2(x)-\psi(x)+x=0$ for every $x\in Q$.
\end{itemize}
\end{proposition}

\begin{remark}
In a Moufang loop we have $(x+y)+z=0$ if and only if $x+(y+z)=0$, so it is not necessary to specify the order of addition in the expression $\psi^2(x)-\psi(x)+x$ above.
\end{remark}

\begin{corollary}\label{Cr:Mend}
Let $Q=(Q,+)$ be a commutative Moufang loop and let $\psi\in\CO{Q}$. The corresponding distributive quasigroup $\mathcal Q(Q,+,id-\psi,\psi,0)$ is:
\begin{itemize}
	\item[(i)] Steiner if and only if $Q$ has exponent $3$ and $\hat\psi=0$;
	\item[(ii)] Mendelsohn if and only if $\hat\psi^2(x)-3\hat\psi(x)+3x=0$ for every $x\in Q$.
\end{itemize}
\end{corollary}

\begin{proof}
Part (i) is obvious. For (ii), we calculate $\psi^2-\psi+id = (-id+\hat\psi)^2-(-id+\hat\psi)+id = \hat\psi^2-3\hat\psi+3id$.
\end{proof}

In particular, if a commutative Moufang loop $Q$ has exponent 3 then the corresponding distributive quasigroup is Steiner if and only if $\hat\psi=0$, and it is Mendelsohn if and only if $\hat\psi^2=0$.

The classification of the respective quasigroups directly translates into the classification of the corresponding triple systems.
Non-medial distributive Mendelsohn quasigroups were enumerated up to order $3^4$ in \cite{DGMOS}, and non-medial distributive Steiner quasigroups were enumerated up to order $3^6$ in \cite{Ben-dq}. In the present paper, we extend the classification in the Mendelsohn case to order $3^5$.

\section{The classification algorithm}\label{Sc:Classification}

\subsection{Outline of the algorithm}

Theorem \ref{Th:Kepka-iso} suggests the following algorithm for the classification of centrally affine quasigroups over a given commutative Moufang loop $Q=(Q,+)$. We calculate the set $J\CA(Q)\times J\CA(Q)\times Z(Q)$ and filter it subject to the equivalence induced by the condition in Theorem \ref{Th:Kepka-iso}. To obtain trimedial quasigroups, we consider only triples $(\varphi,\psi,c)$ satisfying $\varphi\psi=\psi\varphi$. To obtain distributive quasigroups, we consider only triples $(\varphi,\psi,c)$ satisfying $c=0$ and $\varphi+\psi=id$.

To complete the classification for a fixed order $n$, it suffices to consider the disjoint union of the classifications obtained for each commutative Moufang loop of order $n$ because isomorphic centrally affine quasigroups have isomorphic underlying loops; see Theorem \ref{Th:Kepka-iso}. To obtain non-medial quasigroups, we consider only nonassociative loops; see Lemma \ref{Lm:AffineMedial}.

Essentially the same idea was used in \cite{KBL,KN} to classify trimedial and distributive quasigroups of order $3^4=81$ by hand. Manual classification is out of the question for order~$3^5$, and even straightforward computer calculation is insufficient since the size of the set $J\CA(Q)\times J\CA(Q)\times Z(Q)$ is of the magnitude $10^8$ for some of the loops under consideration.

In the rest of this section we describe how to speed up the algorithm.

\subsection{Calculating automorphism groups}

Recall that all six commutative Moufang loops of order $243$ were constructed by Kepka and N\v{e}mec \cite{KN}. Moufang loops of order $81$ were classified by Nagy and Vojt\v{e}chovsk\'y in \cite{NagyVojtechovsky64and81}, and Moufang loops of order $243$ were classified by Slattery and Zenisek in \cite{SlatteryZenisek}. The $71$ nonassociative Moufang loops of order $243$ can be found in the \texttt{LOOPS} \cite{loops} package for \texttt{GAP} \cite{GAP} and can be obtained by calling \texttt{MoufangLoop(243,$i$)}. The six nonassociative commutative Moufang loops correspond to the indices $i\in\{1,2,5,56,57,67\}$.

The default method in \texttt{LOOPS} for calculating automorphism groups of loops is powerful enough to calculate automorphism groups of Moufang loops of order $81$ and even of some loops of order $243$. We adopted the default algorithm, made a better use of global variables and ran it with different choices of generators (to which the algorithm is highly sensitive). We succeeded in calculating the automorphism groups for the six commutative Moufang loops of order $243$. The longest calculation, for \texttt{MoufangLoop(243,5)}, took several hours.

\subsection{Calculating central and $J$-central automorphisms}

We do not calculate the sets $\CA(Q)$, $J\CA(Q)$ and $\CO{Q}$ directly by filtering $\aut(Q)$ because $\aut(Q)$ can be too large. Our approach is based on the following observation.

\begin{lemma}\label{Lm:Referee}
Let $Q$ be a loop and $H$ a subgroup of $\mathrm{Aut}(Q)$ containing $\CA(Q)$. Then $\CA(Q)$ is the kernel of the natural action of $H$ on $Q/Z(Q)$.
\end{lemma}
\begin{proof}
An automorphism $\alpha\in H$ is in the kernel of the action if and only if $\alpha(Z(Q)+x) = Z(Q)+\alpha(x)$ is equal to $Z(Q)+x$ for every $x\in Q$, which says precisely that $\alpha\in\CA(Q)$.
\end{proof}

We can apply Lemma \ref{Lm:Referee} to $H=\mathrm{Aut}(Q)$, which has been obtained above. However, it is possible to calculate $\CA(Q)$ faster using a proper subgroup $H$ of $\mathrm{Aut}(Q)$ as follows.

The standard algorithm for calculating automorphisms of a given algebraic structure attempts to extend a partial map defined on a fixed generating set into an automorphism, while employing various isomorphism invariants to restrict possible images of the generators. Let $X$ be a set of generators of a loop $Q$. Whenever a choice is being made for the image of $x\in X$, we restrict the choice to the coset $Z(Q)+x$. Since we enforce this condition only for generators, the algorithm can yield a subgroup $H$ of $\mathrm{Aut}(Q)$ properly containing $\CA(Q)$. Lemma \ref{Lm:Referee} then allows us to calculate the actual group $\CA(Q)$ as the kernel of the action of $H$.

Having $\CA(Q)$ at our disposal, we can easily calculate the coset $J\CA(Q)$, and filter its elements to obtain $\CO{Q}$.


To finish the classification of \emph{distributive} quasigroups, various subgroups $U$ of $\mathrm{Aut}(Q)$ can be used to filter $\CO{Q}$ up to conjugacy in $U$ (which makes sense thanks to Lemma \ref{Lm:conj}). This is not necessarily as powerful as the conjugacy in the entire group $\mathrm{Aut}(Q)$, but it reduces the number of elements of $\CO{Q}$ to be considered in the final stage, where we employ the entire $\mathrm{Aut}(Q)$ to finish the classification. In our implementation, we used for $U$ the pointwise stabilizer of $Z(Q)$ in $\aut(Q)$.

\subsection{Handling the action on $J\CA(Q)\times J\CA(Q)\times Z(Q)$}

For trimedial quasigroups, we must find a way to handle the equivalence on $J\CA(Q)\times J\CA(Q)$ and the relation between $c_1$ and $c_2$ in the isomorphism test of Theorem \ref{Th:Kepka-iso}.

Consider any group $G$ and a subset $X\subseteq G$ closed under conjugation in $G$. (Later we will take $G=\mathrm{Aut(Q)}$ and $X=J\CA(Q)$, cf. Lemma \ref{Lm:conj}.) Then $G$ acts on $X\times X$ by simultaneous conjugation in both coordinates, i.e., $(\alpha,\beta)^\gamma = (\alpha^\gamma,\beta^\gamma)$. To calculate orbits on $X\times X$, we take advantage of the following well-known result.

\begin{lemma}\label{Lm:Orbits}
Let $G$ be a group acting on a set $X$. Let $O$ be a complete set of orbit representatives of the action, and for every $x\in O$ let $O_x$ be a complete set of orbit representatives of the action of the stabilizer $G_x$ on $X$. Then
\begin{displaymath}
    \{(a,b):a\in O,\,b\in O_a\}
\end{displaymath}
is a complete set of orbit representatives of the action of $G$ on $X\times X$ given by $(x,y)^g = (x^g,y^g)$.
\end{lemma}
\begin{proof}
For every $(x,y)\in X\times X$ there is a unique $a\in O$ and some $z\in X$ such that $(x,y)$ and $(a,z)$ are in the same orbit. For a fixed $a\in O$ and some $u$, $v\in X$, we have $(a,u)$ in the same orbit as $(a,v)$ if and only if $u$, $v$ belong to the same orbit of $G_a$.
\end{proof}

\begin{lemma}\label{Lm:Action}
Let $Q$ be a commutative Moufang loop, let $A=\aut(Q)$, and let $\alpha$,~$\beta\in J\CA(Q)$. Then $C_A(\alpha)\cap C_A(\beta)$ acts naturally on $Z(Q)/\img(id-(\alpha+\beta))$.
\end{lemma}
\begin{proof}
Let $I=\img(id-(\alpha+\beta))$. First, we note that $I\leq Z(Q)$. Indeed, for every $x\in Q$, we have
\begin{displaymath}
    x-(\alpha(x)+\beta(x)) = x-((-x+\hat\alpha(x))+(-x+\hat\beta(x))) = 3x-(\hat\alpha(x)+\hat\beta(x))\in Z(Q),
\end{displaymath}
because $3x\in Z(Q)$ and $\alpha$, $\beta$ are $J$-central.

It remains to show that for every $\gamma\in C_A(\alpha)\cap C_A(\beta)$ the mapping $u+I\mapsto \gamma(u)+I$ is well-defined. Suppose that $u+I=v+I$ for some $u$, $v\in Z(Q)$. Then $u=v+(x-(\alpha(x)+\beta(x)))$ for some $x\in Q$, and we have
\begin{align*}
    \gamma(u) &= \gamma(v) + (\gamma(x) - (\gamma\alpha(x)+\gamma\beta(x)))\\ &= \gamma(v) + (\gamma(x) - (\alpha\gamma(x) + \beta\gamma(x)))\\
    &= \gamma(v) + (id-(\alpha+\beta))(\gamma(x)) \in \gamma(v) + I,
\end{align*}
finishing the proof.
\end{proof}

We can now reformulate Theorem \ref{Th:Kepka-iso} so that it can be used directly in the enumeration of centrally affine quasigroups over a given commutative Moufang loop. (A similar theorem for abelian groups was obtained by Dr\'apal \cite[Theorem 3.2]{Dra} and used as an enumeration tool in \cite{SV}.)

\begin{theorem}\label{Th:Alg}
Let $Q$ be a commutative Moufang loop and let $A=\aut(Q)$. The isomorphism classes of centrally affine quasigroups over $Q$ (resp. trimedial quasigroups over $Q$) are in one-to-one correspondence with the elements of the set
\begin{displaymath}
    \{(\varphi,\psi,c):\varphi\in X,\,\psi\in Y_\varphi,\,c\in Z_{\varphi,\psi}\},
\end{displaymath}
where
\begin{itemize}
	\item $X$ is a complete set of orbit representatives of the conjugation action of $A$ on $J\CA(Q)$;
	\item $Y_\varphi$ is a complete set of orbit representatives of the conjugation action of $C_A(\varphi)$ on $J\CA(Q)$ (resp. on $J\CA(Q)\cap C_A(\varphi)$), for every $\varphi\in X$;
	\item $Z_{\varphi,\psi}$ is a complete set of orbit representatives of the natural action of $C_A(\varphi)\cap C_A(\psi)$ on $Z(Q)/\img(id-(\varphi+\psi))$.
\end{itemize}
\end{theorem}

\begin{proof}
Consider the equivalence relation of $J\CA(Q)\times J\CA(Q)\times Z(Q)$ implicitly defined by Theorem \ref{Th:Kepka-iso}. By Lemma \ref{Lm:Orbits}, it remains to describe when two triples $(\varphi,\psi,c_1)$ and $(\varphi,\psi,c_2)$ are equivalent, where $\varphi\in X$, $\psi\in Y_\varphi$ and $c_1$, $c_2\in Z(Q)$.

Let $I=\img(id-(\varphi+\psi))$. Using Lemma \ref{Lm:Action}, for any $\gamma\in\aut(Q)$ we have $c_2 = \gamma(c_1+u)$ for some $u\in I$ if and only if $c_2\in \gamma(c_1+I) = \gamma(c_1)+I$, which is equivalent to $c_2+I = \gamma(c_1)+I = \gamma(c_1+I)$.
\end{proof}

\section{Results}\label{Sc:Results}

\subsection{Detailed results for order $243$}

The sizes of the various sets of automorphisms encountered during the enumeration can be found in Table \ref{t:aut}. Here $X/G$ denotes the number of orbits of the action of a group $G$ on a set $X$ (where the action is as described above). The loop notation $n/k$ refers to \texttt{MoufangLoop(n,k)}.

\begin{table}[ht]
\begin{displaymath}
\begin{array}{|r|rrrrrr|}
    \hline
    Q                                                                   &   243/1       &   243/2   &   243/5   &   243/56      &   243/57      &   243/67\\
    \text{exponent of }Q                                                &   9           &   27      &   9       &   3           &   9           &   9\\
    Z(Q)       	                                            						&   C_3^2     &   C_9  &   C_3^2  &   C_3^2  &  C_3^2   &   C_9\\
    \text{size of $A = \aut(Q)$}                                                   &   629856     &   34992  &   78732  &   49128768  &   1889568   &   909792\\
    \hline
    |\CA(Q)|=|J\CA(Q)|                                                  &   729         &   81      &   729     &   4374       &   4374       &   81\\
    |J\CA(Q)/A|                          &   16          &   12      &   38      &   8           &   18          &   6\\
    |J\CA(Q)^2/A|    &   1827       &   207     &   11061  &   283         &   2146       &   54\\
    |(J\CA(Q)^2\times Z(Q))/A| &   2310       &   288     &   13056  &   375         &   2537       &   114\\
		\hline
    |\CO{Q}|                                                  &   729         &   81      &   729     &   2187       &   2187       &   81\\
    |(\CO{Q})/A|                   &   16          &   12      &   38      &   6           &   14          &   6\\
    \hline
\end{array}
\end{displaymath}
\caption{Sizes of various subsets of automorphisms that appear in the classification.}
\label{t:aut}
\end{table}

\subsection{Enumeration}

Let $c(Q)$ denote the number of centrally affine quasigroups, $t(Q)$ the number of trimedial quasigroups, $d(Q)$ the number of distributive quasigroups, $dM(Q)$ the number of distributive Mendelsohn quasigroups, and $dS(Q)$ the number of distributive Steiner quasigroups over a loop $Q$, up to isomorphism.

Table \ref{t:main1} displays these numbers for every nonassociative commutative Moufang loop of order $81$ and $243$. The entries for order 81 can be found already in \cite{DGMOS,KBL,KN} and have been independently verified by our calculations. The entries in the last row can be found in \cite{Ben-dq} and have also been independently verified. The remaining entries for order 243 are new. Since all the commutative Moufang loops in the table are nonassociative, the corresponding quasigroups are non-medial by Lemma \ref{Lm:AffineMedial}.

\begin{table}[ht]
\[
\begin{array}{|r|rr|rrrrrr|}   \hline
Q			  & 81/1 & 81/2 & 243/1 & 243/2 & 243/5 & 243/56 & 243/57 & 243/67 \\\hline
c(Q)   & 8 &  27 &  2310 & 288 &  13056 &  375 & 2537 &   114    \\
t(Q)   & 8 &  27 &  2310 & 288 &  13056 &  165 & 1071 &   114    \\
d(Q)   & 2 &   4 &  16 &  12 &  38 & 6& 14& 6   \\
dM(Q)  & 2 &   0 &   0 &   0 &   0 & 5&  1& 0 \\
dS(Q)  & 1 &   0 &   0 &   0 &   0 & 1&  0& 0 \\\hline
\end{array}
\]
\caption{Enumeration of various classes of centrally affine quasigroups over a given commutative Moufang loop.}
\label{t:main1}
\end{table}

Note that the entries $c(Q)$ for loops of order $243$ in Table \ref{t:main1} are precisely the entries in the $8$th row of Table \ref{t:aut}, as explained by Theorem \ref{Th:Alg}.

In Table \ref{t:main2} we summarize the results of Table \ref{t:main1} by order, and we use a notation analogous to that of Table \ref{t:main1}. For instance, $t(n)$ denotes the number of \emph{non-medial} trimedial quasigroups of order $n$ up to isomorphism. Note that we have \emph{not} enumerated non-medial centrally affine quasigroups of order $243$, since this would require also the enumeration of all quasigroups $\mathcal Q(Q,+,\varphi,\psi,c)$, where $(Q,+)$ is an abelian group of order $243$ and $\varphi$, $\psi$ are \emph{non-commuting} automorphisms of $(Q,+)$; a difficult task (see \cite{SV}).

\begin{table}[ht]
\[
\begin{array}{|r|rrrr|} \hline
n       & 3^3 & 3^4 & 3^5 & 3^6 \\\hline
t(n)   &    0 &  35 & {17004}    &?     \\
d(n)   &    0 &   6 & {92} &?     \\
dM(n)  &    0 &   2 & {6} &?     \\
dS(n)  &    0 &   1 &   1 & 3   \\\hline
\end{array}
\]
\caption{Enumeration of various classes of non-medial quasigroups for a given order.}
\label{t:main2}
\end{table}

\subsection{Explicit constructions}

Detailed results of the enumeration, including arithmetic forms for all the quasigroups, can be obtained from the third author upon request.

To present a sample of the detailed results, we now give explicit formulas for all elements of $\CO{Q}$ up to conjugacy in $\aut(Q)$, where $Q$ is \texttt{MoufangLoop(243,$i$)} with $i=56$ or $i=57$ (these are the two directly decomposable non-associative commutative Moufang loops of order 243). The corresponding distributive quasigroups can be obtained readily using Proposition \ref{Cr:BelousovSoublin}. In particular, we obtain an explicit description of all non-affine distributive Mendelsohn triple systems of order 243.

\begin{example}
Consider $Q=\texttt{MoufangLoop(243,56)}=\texttt{MoufangLoop(81,1)}\times\Z_3$.
According to \cite{KN}, the loop $\texttt{MoufangLoop(81,1)}$ is isomorphic to $(\Z_3^4,+)$, where
\begin{displaymath}
    (a_1,b_1,c_1,d_1)+(a_2,b_2,c_2,d_2) = (a_1+a_2+(d_1-d_2)(b_1c_2-c_1b_2),b_1+b_2,c_1+c_2,d_1+d_2).
\end{displaymath}
The associator subloop is $A(Q)=\Z_3\times 0\times 0\times 0\times 0$ and the center is $Z(Q)=\Z_3\times 0\times 0\times 0\times \Z_3$.

The elements of $\CO{Q}$ up to conjugacy by $\aut(Q)$ are given by the following six endomorphisms into the center:
\begin{align*}
  \hat\psi_1: (a,b,c,d,e)\mapsto (0,0,0,0,0),\qquad & \hat\psi_2: (a,b,c,d,e)\mapsto (b,0,0,0,0),\\
  \hat\psi_3: (a,b,c,d,e)\mapsto (e,0,0,0,0),\qquad & \hat\psi_4: (a,b,c,d,e)\mapsto (0,0,0,0,b),\\
  \hat\psi_5: (a,b,c,d,e)\mapsto (b,0,0,0,c),\qquad & \hat\psi_6: (a,b,c,d,e)\mapsto (e,0,0,0,b).
\end{align*}
It is straightforward to check that each of these mappings is an endomorphism into the center with a unique fixed point, and that all $id-\psi_i = 2id-\hat\psi_i$ are permutations. By Lemma \ref{Lm:alphahat}, $\psi_i\in\CO{Q}$ for every $i$.

To check that the six mappings are pairwise non-conjugate, we use the following criterion: \emph{Let $\alpha\in\mathrm{End}(Q)$ and $\xi\in\aut(Q)$. If $H$ is a characteristic subloop of $Q$, we have $\alpha^\xi(H)=\xi\alpha(H)$. If both $H$ and $\alpha(H)$ are characteristic subloops of $Q$ then $\alpha(H)=\alpha^\xi(H)$.} Now observe that:
\begin{itemize}
 \item $\im(\hat\psi_1)=0$,
 \item $\im(\hat\psi_2)=A(Q)$ and $\hat\psi_2(Z(Q))=0$,
 \item $\im(\hat\psi_3)=A(Q)$ and $\hat\psi_3(Z(Q))\neq0$,
 \item $\im(\hat\psi_4)$ is neither $A(Q)$, nor $Z(Q)$,
 \item $\im(\hat\psi_5)=Z(Q)$ and $\hat\psi_5(Z(Q))=0$,
 \item $\im(\hat\psi_6)=Z(Q)$ and $\hat\psi_6(Z(Q))\neq0$.
\end{itemize}
\end{example}

\begin{example}
Consider $Q=\texttt{MoufangLoop(243,57)}=\texttt{MoufangLoop(81,2)}\times\Z_3$.
According to \cite{KN}, the loop $\texttt{MoufangLoop(81,2)}$ is isomorphic to $(\Z_3^2\times\Z_9,+)$, where
\begin{displaymath}
    (a_1,b_1,c_1)+(a_2,b_2,c_2) = (a_1+a_2,b_1+b_2,c_1+c_2+3(c_1-c_2)(a_1b_2-b_1a_2)).
\end{displaymath}
The associator subloop is
$A(Q)=0\times 0\times 3\Z_9\times 0$ and the center is $Z(Q)=0\times 0\times 3\Z_9\times\Z_3$.

The elements of $\CO{Q}$ up to conjugacy by $\aut(Q)$ are given by the following endomorphisms into the center:
\begin{align*}
  \hat\psi_1: (a,b,c,d)\mapsto (0,0,0,0),\qquad &
  \hat\psi_2: (a,b,c,d)\mapsto (0,0,3c,0),\\
  \hat\psi_3: (a,b,c,d)\mapsto (0,0,6c,0),\qquad &
  \hat\psi_4: (a,b,c,d)\mapsto (0,0,3d,0),\\
  \hat\psi_5: (a,b,c,d)\mapsto (0,0,3a,0),\qquad &
  \hat\psi_6: (a,b,c,d)\mapsto (0,0,0,a),\\
  \hat\psi_7: (a,b,c,d)\mapsto (0,0,0,c \bmod 3),\qquad &
  \hat\psi_8: (a,b,c,d)\mapsto (0,0,3a,b),\\
  \hat\psi_9: (a,b,c,d)\mapsto (0,0,3c,a),\qquad &
  \hat\psi_{10}: (a,b,c,d)\mapsto (0,0,6c,a),\\
  \hat\psi_{11}: (a,b,c,d)\mapsto (0,0,3a,c \bmod 3),\qquad &
  \hat\psi_{12}: (a,b,c,d)\mapsto (0,0,3d,a),\\
  \hat\psi_{13}: (a,b,c,d)\mapsto (0,0,3d,c \bmod 3),\qquad &
  \hat\psi_{14}: (a,b,c,d)\mapsto (0,0,3d,2c \bmod 3).
\end{align*}
Again, it is straightforward to check that the corresponding mappings $\psi_i$ belong to $\CO{Q}$. To show that they are pairwise non-conjugate, first notice that $\hat\psi_1=0$, $\hat\psi_2=3id$ and $\hat\psi_3=6id$, so they commute with any automorphism. To distinguish the remaining mappings, consider also the characteristic subloop
$B=\{x\in Q:\,x^3=1\}=\Z_3\times\Z_3\times 3\Z_9\times \Z_3$ and observe that
\begin{itemize}
  \item $\im(\hat\psi_i)=A(Q)$ iff $i=4,5$; here $\hat\psi_5(Z(Q))=0$ but $\hat\psi_4(Z(Q))\neq 0$;
  \item $\im(\hat\psi_i)$ is of order 3 but not $A(Q)$ iff $i=6,7$; here $\hat\psi_7(B)=0$ but $\hat\psi_6(B)\neq 0$,
  \item $\im(\hat\psi_i)=Z(Q)$ for $i=8,\dots,14$;
\begin{itemize}
  \item $\hat\psi_i(Z(Q))=0$ for $i=8,9,10,11$, but
\begin{itemize}
	\item $\hat\psi_{8}(B)=Z(Q)$,
	\item $\hat\psi_{11}(B)=A(Q)$,
	\item both $\hat\psi_{9}(B),\hat\psi_{10}(B)$ have order 3, $\neq A(Q)$;
	we have $\hat\psi_{9}=\hat\psi_{2}+\hat\psi_{6}$  and if there existed
	$\xi$ such that $\hat\psi_{9}^\xi=\hat\psi_{10}$ then $\hat\psi_{6}^\xi=\hat\psi_{10}-\hat\psi_{2}=\hat\psi_{9}$
	which is impossible;
\end{itemize}
  \item $\hat\psi_i(Z(Q))=A(Q)$ for $i=12,13,14$, but
\begin{itemize}
	\item $\hat\psi_{12}(B)=Z(Q)$,
	\item $\hat\psi_{13}(B)=\hat\psi_{14}(B)=A(Q)$; they cannot be conjugate, because their squares, $\hat\psi_{13}^2=\hat\psi_2$ and $\hat\psi_{14}^2=\hat\psi_3$, are not.
\end{itemize}
\end{itemize}
\end{itemize}
\end{example}

Which of these quasigroups transform into distributive Mendelsohn triple systems? According to Corollary \ref{Cr:Mend}:
\begin{itemize}
\item for $Q=\texttt{MoufangLoop(243,56)}$ whose exponent is 3, these are precisely the mappings $\hat\psi_i$ with $\hat\psi_i^2=0$, which is the case for $i=1,2,3,4,5$.
\item for $Q=\texttt{MoufangLoop(243,57)}$, since $3Z(Q)=0$, the equation is equivalent to ${\hat\psi_i}^2=-3id$, which is satisfied only for $i=14$.
\end{itemize}
Using Proposition \ref{Cr:BelousovSoublin}, the triple system $(V,B)$ corresponding to the pair $(Q,\hat\psi)$ is defined by
\begin{displaymath}
     V=Q \quad\text{and}\quad B=\{(x,\,y,\,2x-y+\hat\psi(y-x)):\ x,\,y\in Q\}.
\end{displaymath}

\subsection{Commuting central automorphisms}

Upon inspection of Table \ref{t:main1}, we see that in many small nonassociative commutative Moufang loops $Q$, any two $J$-central automorphisms of $Q$ commute. This is partly explained by Proposition \ref{Pr:affine_are_trimedial}.

\begin{lemma}\label{Lm:AuxCommute}
Let $Q$ be a commutative Moufang loop and let $\varphi$, $\psi$ be $J$-central automorphisms of $Q$. Then $\varphi\psi=\psi\varphi$ if and only if $\hat\varphi\hat\psi = \hat\psi\hat\varphi$.
\end{lemma}
\begin{proof}
We must proceed carefully since the addition of mappings on $Q$ is not necessarily an associative operation. However, for any $\alpha\in\CA(Q)$ and $\beta$, $\gamma\in\aut(Q)$ we have $\hat\alpha+(\beta+\gamma) = (\hat\alpha+\beta) + \gamma$ because $\mathrm{Im}(\hat\alpha)\subseteq Z(Q)$ by Lemma \ref{Lm:JCentral}. In particular, we have
\begin{equation}\label{Eq:Aux}
    \hat\psi+\varphi = \hat\psi+\hat\varphi-id = \hat\varphi + \hat\psi-id = \hat\varphi+\psi.
\end{equation}
Now, $\hat\varphi\hat\psi = (id+\varphi)\hat\psi = \hat\psi + \varphi\hat\psi = \hat\psi + \varphi + \varphi\psi$ and, by symmetry, $\hat\psi\hat\varphi = \hat\varphi+\psi+\psi\varphi$. Thanks to \eqref{Eq:Aux}, we see that $\varphi$ and $\psi$ commute if and only if $\hat\varphi$ and $\hat\psi$ commute.
\end{proof}

\begin{proposition}\label{Pr:affine_are_trimedial}
Let $Q$ be a nonassociative commutative Moufang loop of order a power of~$3$ such that $Z(Q)$ is cyclic and $Q/Z(Q)$ is associative. Then any two $J$-central automorphisms of $Q$ commute.
\end{proposition}
\begin{proof}
Let $\varphi$, $\psi$ be $J$-central automorphisms of $Q$. By Lemma \ref{Lm:AuxCommute}, it suffices to show that $\hat\varphi\hat\psi=\hat\psi\hat\varphi$.

By Lemma \ref{Lm:alphahat}, $\hat\varphi$ and $\hat\psi$ are endomorphism into $Z(Q)$. Any endomorphism into $Z(Q)$ has all associators $(x+(y+z))-((x+y)+z)$ in its kernel, and thus vanishes on the associator subloop $A(Q)$. Since $Z(
Q)$ is cyclic, there are integers $a$, $b$ such that $\hat\varphi(z)=az$, $\hat\psi(z)=bz$ for every $z\in Z(Q)$.

By our assumption, $Q/Z(Q)$ is associative and $0<A(Q)$. Thus $0<A(Q)\le Z(Q)$ and the restriction of each of $\hat\varphi$, $\hat\psi$ onto $Z(Q)$ has nontrivial kernel. Since $|Z(Q)|$ is a power of $3$, it follows that $3$ divides $a$ and $b$. Then $ax$, $bx\in Z(Q)$ for every $x\in Q$, and we calculate
\begin{displaymath}
    \hat\varphi\hat\psi(x) = a\hat\psi(x) = \hat\psi(ax) = bax = abx = \hat\varphi(bx) = b\hat\varphi(x) = \hat\psi\hat\varphi(x)
\end{displaymath}
for every $x\in Q$.
\end{proof}

Every commutative Moufang loop of order $\leq3^5$ is centrally nilpotent of class at most two \cite[Lemma 1.6]{KN}. Both of the nonassociative commutative Moufang loops of order $3^4$ have cyclic centers, and so do two of the six nonassociative Moufang loops of order $3^5$ (see Table~\ref{t:aut}). Proposition \ref{Pr:affine_are_trimedial} therefore applies to these loops. However, Proposition \ref{Pr:affine_are_trimedial} does not tell the whole story, as there are commutative Moufang loops of order $3^5$ that have a non-cyclic center, yet any two of its $J$-central automorphisms commute.

\section*{Acknowledgement}

We thank the anonymous referees for useful comments, especially for pointing out that the $J$-central automorphisms of a loop $Q$ with the automorphic inverse property coincide with the coset $J\CA(Q)$, and that the group $\CA(Q)$ can be calculated as in Lemma \ref{Lm:Referee}, which is more elegant and more efficient compared to our original approach.

We also thank Prof. Anthony Evans for historical comments on orthomorphisms.

\end{document}